\documentclass[english]{smfart}
\usepackage[francais,english]{babel}
\usepackage{smfthm}
\usepackage{amssymb}
\usepackage{euscript}
\usepackage[all]{xypic}
\newcommand{\scr}[1]{\EuScript{#1}}
\newcommand{\C}{\mathbb{C}}
\newcommand{\Z}{\mathbb{Z}}
\newcommand{\Q}{\mathbb{Q}}
\renewcommand{\H}{\mathbb{H}}
\newcommand{\OO}{\scr{O}}
\newcommand{\SL}{\mathrm{SL}}
\newcommand{\an}{\mathrm{an}}
\newcommand{\ord}{\mathrm{ord}}
\newcommand{\ip}[1]{\langle #1 \rangle}
\newcommand{\Tate}{\underline{\mathrm{Tate}}}
\renewcommand{\div}{\mathrm{div}}

\author{Nick Ramsey} \address{Department of Mathematics, 
  University of Michigan}
\email{naramsey@umich.edu} 
\title{Geometric and $p$-adic Modular Forms of Half-Integral Weight}

\begin{document}
\frontmatter
\maketitle
\tableofcontents

\section{Introduction}

The aim of this article is to generalize the notion of a classical
(holomorphic) modular form of half-integral weight to rings other than
$\C$.  In particular we define, for any positive integers $N$ and $k$,
the space $M_{k/2}(4N,R)$ of modular forms of weight $k/2$ and level
$4N$ over the ring $R$, for any ring $R$ in which the level $4N$ is
invertible.  The definition is geometric in nature in the sense that
it involves the modular curves $X_1(4N)$, and admits a modification
to a $p$-adic theory by ``deleting the supersingular locus''.  

The author's primary motivation for studying such objects lies in
their relation to (central) special values of modular $L$-functions.
This remarkable relationship, discovered by Waldspurger, expresses
these values (appropriately normalized) in terms of \emph{squares} of
the Fourier coefficients of modular forms of half-integral weight.
Investigating the $p$-adic nature of modular forms of half-integral,
and in particular the nature of $p$-adic families of modular forms of
half-integral weight, represents a first step in one approach to
understanding the $p$-adic variation of (square roots of) these
$L$-values.  Eying the integral weight case, one expects the notion of
overconvergence to play a central role in questions about families of
modular forms of half-integral weight.  This, in turn, is currently
best understood in geometric terms.

In the integral weight case, a geometric theory of modular forms is
achieved by considering sections of tensor powers of the line bundle
$\omega$ which is the push-forward of the sheaf of relative
differentials from the universal elliptic curve over $X_1(N)$
(supposing $N\geq 5$).  Lacking such a ``natural'' bundle to work
with, we resort to an alternative technique which exploits
well-understood modular forms of weight $k/2$ and reduce (by
division) to consideration of certain rational functions on $X_1(4N)$.
This method is reminiscent of the use of the Eisenstein family to
define families of $p$-adic modular forms of integral weight (as in
\cite{coleman} and \cite{colmaz}).

Let $$\theta(\tau) = \sum_{n\in\Z} q^{n^2}, \ \ \ \ \ q = e^{2\pi
  i\tau}$$
denote the usual Jacobi theta function.  It is well-known
that $\theta$ obeys a transformation law with respect to the group
$$\Gamma_0(4) = \left\{\left.\left(\begin{matrix} a & b \\ c &
        d\end{matrix}\right)\in \SL_2(\Z)\right| 4|c \right\},$$
namely
\begin{equation}\label{eqn:thetatrans}
\theta\left(\frac{a\tau+b} {c\tau +d}\right) =
\varepsilon_d^{-1}\left(\frac{c}{d}\right)(c\tau+d)^{1/2}\theta(\tau),\
\mathrm{for\ all}\ \left(\begin{matrix} a & b \\ c &
    d\end{matrix}\right)\in \Gamma_0(4),$$ where $$\epsilon_d =
\left\{ \begin{matrix} 1 & d\equiv 1\pmod{4} \\ i & d \equiv -1
    \pmod{4}\end{matrix}\right.
\end{equation}
and the square root is chosen to
have argument in the interval $[-\pi/2,\pi/2)$.  The Jacobi symbol
here is taken with the conventions of \cite{shimura}.

We will take the $\theta^k$ as the ``well-understood'' forms by which
we divide as indicated above.  Indeed, in section \ref{sec:defs} we
single out a family of divisors $\Sigma_{4N,k}$ on the curves
$X_1(4N)$ with the property that the rational functions $F$ on
$X_1(4N)$ such that $F\theta^k$ is a holomorphic modular forms of
weight $k/2$ are exactly those $F$ for which $(F)\geq -\Sigma_{4N,k}$.
The caveat is that $\Sigma_{4N,k}$ has \emph{non-integral} (rational)
coefficients.  The divisor $\Sigma_{4N,k}$ is supported on the cusps
and makes perfect sense as a relative $\Q$-divisor on $X_1(4N)_R$ over
any base $R$ in which $4N$ is invertible.  We will accordingly
identify $M_{k/2}(4N,R)$ with the $R$-module of rational functions $F$
on $X_1(4N)_R$ with $(F)\geq -\Sigma_{4N,k}$.  In other words, if
$\lfloor\Sigma_{4N,k}\rfloor$ is the relative effective Cartier
divisor obtained by taking the floor of the coefficients of
$\Sigma_{4N,k}$, then $$M_{k/2}(4N,R) = H^0(X_1(4N)_R,\OO(\lfloor
\Sigma_{4N,k} \rfloor)).$$
In what follows, we will adopt the suggestive notation
$$H^0(X_1(4N)_R,\Sigma_{4N,k})$$  for the right-hand side of this 
equation.  

If this definition is to be a useful one, it must enjoy some the host
of properties possessed by the integral weight analog.  To this end, we
will define $q$-expansions, prove a $q$-expansion principle, prove
some basic ``change-of-ring'' compatibilities, and furnish a
\emph{geometric} construction of the Hecke operators.  Because of its
central importance, we briefly expound on this last point.  

Let $X_1(N,m)$ denote the modular curve over $\Z[1/(Nm))]$ (roughly)
classifying elliptic curves with a point of order $N$ and a cyclic
subgroup of order $m$ having trivial intersection with that generated by
the point.  Let
$$\pi_1,\pi_2: X_1(N,m)\longrightarrow X_1(N)$$ denote the
degeneracy maps which forget and quotient by, respectively, the cyclic
subgroup of order $m$.  These pairs of maps form the usual Hecke
correspondences on $X_1(N)$, and the geometric manifestation of the
integral weight Hecke operator $T_m$ (for, say, $m$ prime) is
\begin{eqnarray*}
H^0(X_1(N),\omega^k) & \longrightarrow & H^0(X_1(N),\omega^k) \\ 
f & \longmapsto & \pi_{1*}(\pi_2^*f),
\end{eqnarray*}
where we have exploited the canonical identification $$\pi_1^*\omega
\cong \pi_2^*\omega.$$

A little reflection on the definition of a modular form of
half-integral weight given above leads one to look for a map of the
form
\begin{eqnarray*}
  H^0(X_1(4N),\Sigma_{4N,k}) & \longrightarrow &
  H^0(X_1(4N),\Sigma_{4N,k}) \\ F & \longmapsto &
  \pi_{1*}(\pi_2^*F\cdot \Theta)
\end{eqnarray*}
for some rational function $\Theta$ on $X_1(4N,m)$ satisfying
\begin{equation}\label{eqn:thetacondition}
\div(\Theta)\geq \pi_2^*\Sigma_{4N,k} - \pi_1^*\Sigma_{4N,k}.
\end{equation}
Note that this last divisor is of degree $0$, so if such a $\Theta$
exists, then one has equality above. It then follows that $\Theta$ is
necessarily unique up to a constant factor, and the divisor
$\pi_2^*\Sigma_{4N,k} - \pi_1^*\Sigma_{4N,k}$ actually has integral
coefficients.  The following is proven in section \ref{sec:hecke}, and
involves a slightly refined space of forms $M_{k/2}(4N,\chi,R)$ with
Nebentypus character $\chi:(\Z/4N\Z)^\times\to R^\times$.
\begin{theo}
  Suppose $\gcd(4N,m)=1$.  Then there exists a $\Theta$ satisfying
  (\ref{eqn:thetacondition}) if and only if $m$ is a square.  If
  $m=l^2$ for a prime $l$ not dividing $4N$, the corresponding
  endomorphism $T_{l^2}$ of $M_{k/2}(4N,\chi,R)$ (appropriately
  normalized) has the following effect: if $\sum a_nq^n$ is the
  $q$-expansion of $F$ at $\infty$, then the $q$-expansion of
  $T_{l^2}F$ at $\infty$ is $\sum b_nq^n$ where $$b_n = a_{l^2n} +
  \left(\frac{-1}{l}\right)^{(1-k)/2}l^{(k-1)/2-1}
  \left(\frac{n}{l}\right)\chi(l)a_n + l^{k-2}\chi(l^2)a_{n/l^2}.$$
\end{theo}
This theorem furnishes a geometric explanation of the classical fact
that there are no good (i.e. prime to the level) Hecke operators of
non-square index, and shows that the operators obtained by the
procedure above at square indices are the same as the classical ones
(as in \cite{shimura}).  

Let $K$ be a discretely valued subfield of $\C_p$.  In section
\ref{sec:padic} we define $K$-Banach space $M_{k/2}(4N,K,r)$ of
$p$-adic modular forms of weight $k/2$, level $4N$, and growth
condition $r\in[0,1]$, and prove the following theorem.
\begin{theo}
  Let $p$ be a prime number not dividing $4N$ and let $r$ further
  satisfy $r>p^{-1/p(1+p)}$.  There is a completely continuous
  endomorphism $U_{p^2}$ of $M_{k/2}(4N,K,r)$ having the effect $$\sum
  a_nq^n\longmapsto \sum a_{p^2n}q^n$$
  on $q$-expansions at $\infty$.
  Moreover, the norm of $U_{p^2}$ is at most $p^2$.
\end{theo}

The definition of the space $M_{k/2}(4N,K,r)$ is a straightforward
modification of that of $M_{k/2}(4N,R)$ achieved by deleting part of
the supersingular locus.  In particular, it still relies on properties
of the forms $\theta^k$ (equivalently, the divisors $\Sigma_{4N,k}$).
We note that it is perhaps more natural in this $p$-adic context to
use Eisenstein series of half-integral weight in place of $\theta$,
since much of their divisors lie \emph{outside of the range of consideration}.
This point of view is adopted in a forthcoming work investigating
$p$-adic families of modular forms of half-integral weight
(see \cite{hieigencurve} and \cite{rigidshimura}) .\newline

\noindent{\bf Acknowledgments}\newline

This work was done while the author was a graduate student Harvard
University.  He extends his  thanks to his advisor Barry Mazur
for suggesting a problem which led to this work and for being helpful
and encouraging throughout.  He would also like to thank the referee
for pointing out some typos, as well as providing some helpful
suggestions about the level of detail.

\section{Notation and Conventions}\label{sec:notconv}

For an integer $N\geq 4$ and a ring $R$ in which $N$ is
invertible, $Y_1(N)$ will denote the fine moduli scheme classifying
pairs $(E/S,P)$ where $E$ is an elliptic curve over an $R$-algebra $S$
and $P$ is a point of order $N$ (as in \cite{katzmazur}).  For such
$N$, $X_1(N)_R$ will denote the compactified moduli scheme, which is a
fine moduli scheme for the \emph{generalized} $\Gamma_1(N)$ problem if
$N\geq 5$, and coarse moduli scheme if $N=4$.  We will have no
occasion to make use of generalized elliptic curves in what follows.  

Choosing another positive integer $m$ which is invertible in $R$, we
may also form the moduli scheme $Y_1(N,m)_R$ classifying triples
$(E/S,P,C)$ consisting of an elliptic curve $E$ over an $R$-algebra
$S$, a point $P$ of order $N$, and a cyclic subgroup $C$ of order $m$
not meeting that generated by $P$, as well as its compactification
$X_1(N,m)_R$.  All of these moduli schemes are smooth over $R$.

If $K$ is the complex field $\C$ or a $p$-adic field, we will denote
by $X_1(N)_K^\an$ (and similarly $X_1(N,m)_K^\an$, etc.) the
associated complex-analytic or rigid-analytic space, respectively.
Over $\C$, we choose to uniformize $Y_1(N)$ via the map
\begin{eqnarray*}
  \Gamma_1(N)\backslash\H & \stackrel{\sim}{\longrightarrow} &
  Y_1(N)_\C^\an \\ \tau & \longmapsto & \left(E_\tau,\frac{1}{N}\right)
\end{eqnarray*}
where $\H$ denotes the complex upper half-plane and $E_\tau$ denote
the elliptic curve $\C/\ip{1,\tau}$ over $\C$.  Similarly, we choose to
uniformize $Y_1(N,m)$ via the map
\begin{eqnarray*}
  (\Gamma_1(N)\cap \Gamma^0(m))\backslash\H &
  \stackrel{\sim}{\longrightarrow} & 
  Y_1(N,m)_\C^\an \\ \tau & \longmapsto &
  \left(E_\tau,\frac{1}{N},\left\langle
  \frac{\tau}{m}\right\rangle\right).
\end{eqnarray*}
If $\Gamma\subseteq\SL_2(\Z)$ is any subgroup, we will also use the
notation $X(\Gamma)_\C^\an$ for the compactification of the quotient
$\Gamma\backslash\H$, ignoring any algebraic model of this Riemann
surface (i.e. this is not intended to denote the analytification of
anything in this case).

We will also need to consider the moduli spaces $X_1(N,p^n)_R$ in the
case that $p$ is \emph{not} invertible in $R$ (but $N$ is).  By such a
space we will always mean the Katz-Mazur model (see \cite{katzmazur}).
In particular, $X_1(N,p^n)_R$ is not smooth over $R$ and will even
have non-reduced fibers when $n>1$.  We note for later use that
$X_1(N,p^n)_R$ is Cohen-Macaulay over $R$ as its singularities are all
local complete intersections.

We will denote the Tate elliptic curve over $\Z((q))$ by $\Tate(q)$
(see \cite{katz}).  Our conventions concerning the Tate curve differ
from the standard ones as follows.  In the presence of, for example,
level $N$ structure, previous authors (e.g.,\ \cite{katz}) have
preferred to consider the curve $\Tate(q^N)$ over the base $\Z((q))$.
Points of order $N$ on this curve are used to characterize the
behavior of a modular form at the cusps, and are all defined over the
fixed ring $\Z((q))[\zeta_N]$ (where $\zeta_N$ is some primitive
$N^{\small\mathrm{th}}$ root of $1$).  We prefer to fix the curve
$\Tate(q)$ and instead consider \emph{extensions} of the base.  Thus,
in the presence of level $N$ structure, we introduce the formal
variable $q_N$, and \emph{define} $q=q_N^N$.  Then the curve
$\Tate(q)$ is defined over the sub-ring $\Z((q))$ of $\Z((q_N))$ and
all of its $N$-torsion is defined over the ring $\Z((q_N))[\zeta_N]$.
To be precise, the $N$-torsion is given by $$\zeta_N^iq_N^j,\ \ 0\leq
i,j\leq N-1.$$  

The reason for this formal relabeling is that, with these
conventions, the $q$-expansions of modular forms look like they do
complex-analytically if we identify $q_N$ with $e^{2\pi i\tau/N}$.  In
particular, the ``$q$-order'' of vanishing (i.e. the smallest,
generally fractional, power of $q$ appearing in the expansion) at a
cusp does not change when one pulls back through a degeneracy map, even
in the presence of ramification.

For a prime $l$ and an $l^{\small\rm th}$ root of unity $\zeta$ we
will denote the associated quadratic Gauss sum by
$$\mathfrak{g}_l(\zeta) :=
\sum_{a=1}^{l-1}\left(\frac{a}{l}\right)\zeta^a.$$ 

\section{Definitions}\label{sec:defs}

Given that $\theta$ is non-vanishing in $\H$, we must examine the
behavior of $\theta$ at the cusps of $X_1(4)_\C^\an$ in order to
determine the divisor $\Sigma_{4N,k}$ mentioned in the introduction.
There are three such cusps, and using standard formulas for
transforming ``theta functions with characteristics'' one finds that the
$q$-expansions of $\theta$ at these cusps are as follows (defined up
to a constant involving powers of $2$ and roots of unity, except in
the first case).
\begin{equation}\label{tab:table1}
\begin{tabular}{l|l}
  cusp on $X_1(4)^\mathrm{an}$ & $q$-expansion   \\\hline 
  $\infty$ & $\theta_1(q):=\sum_{n\in\Z}q^{n^2}$ \\
  $1/2$ &  $\theta_3(q):=q_4\sum_{n\in\Z}q^{n^2+n}$ \\
  $0$ &  $\theta_2(q):=\sum_{n\in\Z} q_{4}^{n^2}$ 
\end{tabular}
\end{equation}
Recall that in the complex-analytic setting, $q_h$ is shorthand for
 $e^{2\pi i\tau/h}$.  
 
 If $F$ is a rational function on $X_1(4N)_\C^\an$, it follows that
 $F\cdot\theta^k$ is holomorphic if and only if $(F)\geq
 -\Sigma_{4N,k}$, where 
\begin{equation}\label{eqn:sigmadef}
\Sigma_{4N,k} = \sum_{c\sim
   1/2}\frac{kw_c}{4}\cdot c,
\end{equation}
the sum is taken over all cusps $c$ on $X_1(4N)_\C^\an$ mapping to
$1/2$ on $X_1(4)_\C^\an$ under the degeneracy map which multiplies the
point of order $4N$ by $N$, and $w_c$ denotes the width of the cusp
$c$ (which coincides with the ramification index of this map at $c$ since
$1/2\in X_1(4)_\C^\an$ has width $1$).  

To determine the appropriate replacement for $\Sigma_{4N,k}$ on the
algebraic curve $X_1(4N)_R$ ($R$ a $\Z[1/(4N)]$-algebra), we consider
the form $\theta^4$.  By the transformation formula
(\ref{eqn:thetatrans}), this is a modular form of weight $2$ for
$\Gamma_1(4)$.  By GAGA and the $q$-expansion principle, $\theta^4$
furnishes a section of $\omega^2$ on the curve $X_1(4)_{\Z[1/2]}$.
Thus we may evaluate the corresponding rule on pairs $(\Tate(q),P)$
where $P$ runs through the points of order $4$ on $\Tate(q)$ (recall
the conventions Section \ref{sec:notconv}).  There are six such pairs
up to isomorphism, and the values are as follows,
\begin{equation}\label{tab:table2}
\begin{tabular}{l|c|c}
 point on $\Tate(q)$ &cusp on $X_1(4)_\C^\an$ &   $q$-expansion
 \\\hline   
 $\zeta_4$ & $\infty$ &   $\left(\sum_{n\in\Z}q^{n^2}\right)^4$ \\
 $\zeta_4 q_2$ & $1/2$ &   $q\left(\sum_{n\in\Z}q^{n^2+n}\right)^4$ \\
 $q_4$ & $0$ &   $-\frac{1}{4}\left(\sum_{n\in\Z} q_{4}^{n^2}\right)^4$ \\
 $\zeta_4q_4$ & $0$ &
 $-\frac{1}{4}\left(\sum_{n\in\Z}(\zeta_4q_4)^{n^2}\right)^4$ \\
 $\zeta_4^2q_4=-q_4$ & $0$ &
 $-\frac{1}{4}\left(\sum_{n\in\Z}(-q)^{n^2}\right)^4$  
 \\   $\zeta_4^3q_4=-\zeta_4q_4$ & $0$ &
 $-\frac{1}{4}\left(\sum_{n\in\Z}(\zeta_4^3q_4)^{n^2}\right)^4$  
\end{tabular}
\end{equation}
where we have filled in the middle column by comparison with table
(\ref{tab:table1}).  

It is now clear from the table that the cusp associated to the pair
$(\Tate(q),\zeta_4q_2)$ is the correct replacement for the cusp $1/2$
in general.  To simplify notation, we will simply denote the cusp on
$X_1(4)_R$ corresponding to this pair by the symbol $1/2$, and
accordingly \emph{define} a divisor $\Sigma_{4N,k}$ on the curve
$X_1(4N)_R$ by the formula (\ref{eqn:sigmadef}) for any ring $R$ in
which $4N$ is invertible.  

Now that we have a good version of $\Sigma_{4N,k}$ on $X_1(4N)_R$, we
enshrine the identification of the Introduction in a
\emph{definition}.
\begin{defi}
  Let $k$ be an odd positive integer, let $N$ be any positive integer,
  and let $R$ be a ring in which $4N$ is invertible.  A modular form
  of weight $k/2$ and level $4N$ is an element of the $R$-module
  $$M_{k/2}(4N,R) := H^0(X_1(4N)_R,\Sigma_{4N,k}).$$
\end{defi}
We recall from the Introduction that this notation is shorthand for
the collection of rational functions $F$ on $X_1(4N)_R$ with $(F)\geq
-\Sigma_{4N,k}$.  In other words, this is the space of section of the
sheaf $\OO(\lfloor \Sigma_{4N,k}\rfloor)$ associated to the relative
effective Cartier divisor $\lfloor\Sigma_{4N,k}\rfloor$ on
$X_1(4N)_R$.  The reason for carrying around the $\Q$-divisor
$\Sigma_{4N,k}$ instead of uniformly resorting to its floor
$\lfloor\Sigma_{4N,k}\rfloor$ (or equivalently the associated sheaf)
is that the latter do not behave well with respect to degeneracy maps
and the former does.  For example, if $N|M$ and $$\pi:X_1(4M)\to X_1(4N)$$
denotes the map which multiplies the point by $M/N$, then
$$\pi^*\Sigma_{4N,k} = \Sigma_{4M,k}.$$
The precise behavior of
$\Sigma_{4N,k}$ under various degeneracy maps will be of critical
importance in the sequel.

For $d\in (\Z/(4N)\Z)^\times$, let $\ip{d}$ denote the diamond
automorphism of $X_1(4N)$ corresponding to $d$.  Note that
$$\ip{d}^*\Sigma_{4N,k} = \Sigma_{4N,k},\ \ \mathrm{for\ all}\ \ d\in
(\Z/(4N)\Z)^\times.$$  It follows that the group $(\Z/(4N)\Z)^\times$
acts on the space $M_{k/2}(4N,R)$.  For a character
$$\chi:(\Z/(4N)\Z)^\times\to R^\times,$$ we will denote the
$\chi$-isotypic component of $M_{k/2}(4N,R)$ by $M_{k/2}(4N,\chi,R)$,
and refer to such forms as having ``nebentypus $\chi$''.

\section{Properties}

As in the integral weight case, it will be useful to reinterpret the
elements of $M_{k/2}(4N,R)$ as ``rules'' which take test objects as
input and give some output subject to a few restraints.  In
particular, we will take as test objects pairs $(E/S,P)$ where $E$ is
an elliptic curve over an $R$-algebra $S$ and $P$ is a point of order
$4N$ on $E$.  An element $F$ of $M_{k/2}(4N,R)$ can then be identified
with a rule (also called $F$) which takes as input a test object
$(E/S,P)$ and outputs an element of $S$ such that
\begin{itemize}
\item if the test objects $(E/S,P)$ and $(E'/S,P')$ are isomorphic
  over $S$, then $F(E,P) = F(E',P')$,
\item if $\phi:S\to S'$ is any map of $R$-algebras and $(E/S,P)$ is a
  test object over $S$, then $$F((E,P)\times_S S') = \phi(F(E,P)),$$
  and 
\item if $P$ is a point of order $4N$ on $\Tate(q)$, then
  $$F(\Tate(q),P)\in q_{4N}^{-e}\Z[[q_{4N}]]$$ where $$e =
  \frac{4N}{w_c}\ord_c\Sigma_{4N,k} = \left\{\begin{matrix} Nk & c\sim
  1/2 \\ 0 & c\nsim 1/2\end{matrix}\right.$$
\end{itemize}
and $c$ is the cusp associated to $(\Tate(q),P)$.

With this in mind, we can define define the $q$-expansions of a
modular form of half-integral weight by evaluating the corresponding
rule on pairs $(\Tate(q),P)$ for various points $P$ of order $4N$.
By the definition of the space $M_{k/2}(4N,R)$, these $q$-expansions
will not agree with the classical ones, and will in fact be off
exactly by a factor of the corresponding $q$-expansion of $\theta^k$,
and we must adjust accordingly.

To be precise, let $P$ be a point of order $4$ on $\Tate(q)$.  Then,
up to isomorphism, $P$ is one of the points listed in the first column
of Table \ref{tab:table2}.  Up to roots of unity, the corresponding
$q$-expansions of $\theta$ are as follows (see Table \ref{tab:table1}).
\begin{equation}\label{tab:table3}
  \begin{tabular}{l|c}
    $P$ & $q$-expansion of $\theta$ \\ \hline
    $\zeta_4$ & $\sum_{n\in\Z} q^{n^2}$ \\ 
    $\zeta_4q_2$ & $q_4\sum_{n\in\Z} q^{n^2+n}$ \\
    $\zeta_4^kq_4,\ 0\leq k\leq 3$ &
    $\frac{1}{1+\zeta_4^k}\sum_{n\in\Z}\zeta_4^{kn^2}q_4^{n^2}$ 
  \end{tabular}
\end{equation}

\begin{defi}
  Let $F\in M_{k/2}(4N,R)$ and let $P$ be a point of order $4N$ on
  $\Tate(q)$.  Let $\theta_P$ be the $q$-expansion of $\theta$ in the
  above table corresponding to the point $NP$ of order $4$.  The
  $q$-expansion of $F$ at $(\Tate(q),P)$ is
 $$F(\Tate(q),P)\theta_P^k \in \Z((q_{4N}))\otimes
 R[\zeta_{4N}].$$  
\end{defi}
We remark in particular that the definition of $\Sigma_{4N,k}$ exactly
suffices to ensure that all of the $q$-expansions of an $F\in
M_{k/2}(4N,R)$ are in fact in the subring $$\Z[[q_{4N}]]\otimes
R[\zeta_{4N}].$$  That is, such an $F$ is ``holomorphic at the
cusps.''  We will often fix a root of unity $\zeta_{4N}$, and refer to
the $q$-expansion at $(\Tate(q),\zeta_{4N})$ as the ``$q$-expansion at
$\infty$.''  The reason is that if $R=\C$ and we have fixed
$\zeta_{4N}=e^{2\pi i/N}$, then under the uniformizations detailed in
Section \ref{sec:notconv} we recover the classical $q$-expansion at
``$i\infty$.''  

Half-integral weight modular forms enjoy a $q$-expansion principle
much like the integral-weight version.  As in that case, one must
first define modular forms over \emph{modules} as opposed to
\emph{rings}.  For a $\Z[1/(4N)]$-module $K$, one defines the space of
modular forms of weight $k/2$ and level $4N$ with coefficients in $K$
as the collection of rules which associate to each test datum $(E,P)$
over a $\Z[1/(4N)]$-algebra $R$, an element of $R\otimes_{\Z[1/(4N)]}
K$, subject to the usual compatibilities and conditions on the Tate
curve.  Note that, in the case that $K$ is in fact a ring, the space
so defined in canonically isomorphic to the space we have already
defined.  However, in contrast to the case of rings, the
$q$-expansions of elements of $M_{k/2}(4N,K)$ for a
$\Z[1/(4N)]$-module $K$, lie in the ring $$\Z[[q_{4N}]]\otimes_\Z
\Z[1/(4N),\zeta_{4N}]\otimes_{\Z[1/(4N)]} K,$$
In particular, these
$q$-expansion \emph{do not} agree with the ones defined for rings
(even when $K$ is a ring), though one can still deduce useful
information about the ``ring'' $q$-expansions from the ``module''
$q$-expansions.  At any rate, it is in the context of the ``module''
$q$-expansion that the $q$-expansion principle is most naturally stated.
\begin{theo}
  Let $K\subset L$ be an inclusion of $\Z[1/(4N)]$-modules and let
  $F\in M_{k/2}(4N,L)$.  Let $P$ be a points of order $4N$ on
  $\Tate(q)$ and suppose that the $q$-expansion of $F$ at
  $(\Tate(q),P)$ has coefficients in
  $\Z[1/(4N),\zeta_{4N}]\otimes_{\Z[1/(4N)]} K$.  Then $F\in
  M_{k/2}(4N,K)$.  
\end{theo}
\begin{proof}
  Referring to Table \ref{tab:table3}, we see that each of $\theta_P$
  and $\theta_P^{-1}$ has coefficients in $\Z[1/(4N),\zeta_{4N}]$.  It
  follows that the $q$-expansion of $F$ at $(\Tate(q),P)$ has
  coefficients in $\Z[1/(4N),\zeta_{4N}]\otimes_{\Z[1/(4N)]} K$ if and
  only if $F(\Tate(q),P)$ does.  If the latter is true, then $F\in
  M_{k/2}(4N,K)$ by the ordinary $q$-expansion principle in weight $0$
  (the poles are of no consequence).  
\end{proof}

We wrap up this section with a basic result about changes-of-ring.
Any map $R\to S$ of $\Z[1/(4N)]$-algebras induces a map
$$M_{k/2}(4N,R)\otimes_R S \longrightarrow M_{k/2}(4N,S).$$  We would
like conditions under which this map is an isomorphism.  As these are
sections of the sheaf $\OO(\lfloor\Sigma_{4N,k}\rfloor)$, it suffices
by standard base-changing results to see that
$$H^1(X_1(4N)_R,\OO(\lfloor\Sigma_{4N,k}\rfloor) = 0.$$
\begin{theo}
  This vanishing of $H^1$ holds whenever $k\geq 5$.
\end{theo}
\begin{proof}
  By Serre duality it suffices to show that
  $$\deg(\lfloor\Sigma_{4N,k}\rfloor) >\deg(\Omega^1_{X_1(4N)_R}).$$
  Note that the form $\theta^4$ furnishes a section of $\omega^2$ on
  $X_1(4N)_R$ with divisor $\Sigma_{4N,4}$.  Also, by the
  Kodaira-Spencer isomorphism, $$\Omega^1_{X_1(4N)}\cong
  \omega^2(-C)$$ where $C$ is the divisor of cusps.  Thus
  $$\deg(\Omega^1_{X_1(4N)}) = \deg(\omega^2)- |C| <
  \deg(\Sigma_{4N,4})\leq \deg(\lfloor \Sigma_{4N,k}\rfloor)$$
  whenever $k\geq 4$.
\end{proof}
\begin{rema}
  In weights $1/2$ and $3/2$ this inequality of degrees does not hold
  in general, the first counterexamples occurring at levels $20$ and
  $68$, respectively.  The author does not know if base change holds
  in these cases.
\end{rema}

\section{The Modular Units $\Theta_m$}\label{sec:thetam}

Let $m$ be an odd positive integer.  In this section we introduce a
family of rational functions on $X_1(4m)$ whose divisors are supposed
on the cusps.  These functions will later be used to construct the
Hecke operators.

Let $\Theta_m$ be the holomorphic function on
$\H$ defined by
$$\Theta_m(\tau) = \frac{\theta(\tau/m)}{\theta(\tau)}.$$
Note that,
for $$\gamma=\left(\begin{matrix} a & b \\ c & d\end{matrix}\right)\in
\Gamma_1(4)\cap \Gamma^0(m),$$
we have $$\Theta_m|\gamma =
\left(\frac{m}{d}\right) \Theta_m.$$
Thus $\Theta_m$ furnishes a
holomorphic function on $(\Gamma_1(4)\cap \Gamma^1(m))\backslash\H$
which descends through the natural map 
\begin{equation}\label{eqn:natmap}
(\Gamma_1(4)\cap
\Gamma^1(m))\backslash\H\longrightarrow (\Gamma_1(4)\cap
\Gamma^0(m))\backslash\H 
\end{equation}
exactly when $m$ is a square.  Let $d$ denote the extension of this
map to $$X(\Gamma_1(4)\cap \Gamma^1(m))_\C^\an \longrightarrow
X_1(\Gamma_1(4)\cap \Gamma^0(m))_\C^\an = X_1(4,m)_\C^\an.$$

For $N\geq 1$ we define two maps $$\pi_i:Y_1(4N,m)\longrightarrow
Y_1(4N)$$ by the following two transformations of moduli functors
\begin{eqnarray*}
  \pi_1 : (E,P,C) &\longmapsto& (E,P) \\
  \pi_2 : (E,P,C) &\longmapsto& (E/C,P/C).
\end{eqnarray*}
We remark that in the complex-analytic uniformizations of these curves
given in Section \ref{sec:notconv}, these maps are given by
$\pi_1(\tau) = \tau$ and $\pi_2(\tau) = \tau/m$.  

\begin{prop}\label{divtheta}
  The divisor of $\Theta_m$ as a meromorphic function
  on\\$X(\Gamma_1(4)\cap\Gamma^1(m))_\C^\an$ is $$\div(\Theta_m) = 
  (\pi_2\circ d)^*\Sigma_{4,1} - (\pi_1\circ d)^*\Sigma_{4,1}$$
\end{prop}
\begin{proof}
  It is tempting to deduce this directly from Table (\ref{tab:table1})
  and the ``formula'' $\Theta_m = (\pi_2\circ d)^*\theta/(\pi_1\circ
  d)^*\theta$.  The problem is that $\theta$ has not been realized as
  the section of any bundle.  One way around this is to raise both
  sides to the fourth power and use the fact that $$\theta^4\in
  H^0(X_1(4)^\an_\C, \omega^2)=M_2(\Gamma_1(4),\C).$$
  We have
  $$\Theta_m^4(\tau) = \frac{\theta(\tau/m)^4} {\theta(\tau)^4},$$
  so
  that $$\Theta_m^4 = (\pi_2\circ d)^*\theta^4\otimes((\pi_1\circ
  d)^*\theta^4)^{-1}$$ 
  where we have used the canonical identification $$\OO =
  \omega^2\otimes (\omega^2)^{-1}.$$  It follows that
  \begin{eqnarray*}
    4\div_{\OO}\Theta_m &=& \div_{\OO}\Theta_m^4 \\ &=&
    \div_{\omega^2}(\pi_2\circ d)^*\theta^4 -
    \div_{\omega^2}(\pi_1\circ d)^*\theta^4 \\
    &= & (\pi_2\circ d)^*\div_{\omega^2}\theta^4 - (\pi_1\circ d)
    ^*\div_{\omega^2}\theta^4.
  \end{eqnarray*}
The desired result will then follow from the fact that
$\div_{\omega^2}\theta^4= 4\Sigma_{4,1}$ upon division by $4$.  This
  fact is easily read off from  Table (\ref{tab:table2}).\end{proof}

Since $\gcd(4,m)=1$, the groups $\Gamma_1(4)\cap \Gamma^1(m)$ and
$\Gamma_1(4m)$ are conjugate, and the curves $X(\Gamma_1(4)\cap
\Gamma^1(m))^\an_\C$ and  $X_1(4m)^\an_\C$ are isomorphic.  One way to
realize this isomorphism explicitly is via the moduli interpretation of
the latter curve and the map 
\begin{equation}\label{eqn:isom}
X(\Gamma_1(4)\cap
\Gamma^1(m))^\an_\C\longrightarrow X_1(4m)^\an_\C
\end{equation}
given (on noncuspidal points) by $$\tau\longmapsto (E_\tau,P)$$
where
$P$ is the unique point on $E_\tau$ with $mP = 1/4$ and $4P = \tau/m$.
Translating the natural map $d$ into these terms one arrives at the
map
$$d':X_1(4m)^\an_\C\to X_1(4,m)^\an_\C$$ given on noncuspidal points by
$$(E,P)\longmapsto (E,mP,\ip{4P}).$$
We denote the pullback of $\Theta_m$ through the (inverse of) the
isomorphism (\ref{eqn:isom}) by the same name.  It's divisor is then
$$\div(\Theta_m) = (\pi_2\circ d')^*\Sigma_{4,1} - (\pi_1\circ
d')^*\Sigma_{4,1}.$$

Recall that the divisor $\Sigma_{4N,k}$ is defined more generally on
the algebraic curve $X_1(4N)_R$ for any $R$ in which $4N$ is
invertible.  Applying GAGA to $\Theta_m$ thought of as a section of
$$\OO_{X_1(4m)_\C^\an}((\pi_2\circ d')^*\Sigma_{4,1} - (\pi_1\circ
d')^*\Sigma_{4,1}),$$
we see that $\Theta_m$ comes from a section
(also denoted $\Theta_m$) of the corresponding sheaf on the
\emph{algebraic} curve $X_1(4m)_\C$.  Repeating this argument with
$\Theta_m^{-1}$ as a section of
$$\OO_{X_1(4m)_\C^\an}((\pi_1\circ d')^*\Sigma_{4,1} - (\pi_2\circ
d')^*\Sigma_{4,1})$$
furnishes the opposite inequality of the divisor
of $\Theta_m$ and shows that $\Theta_m$ is a rational function on the
algebraic curve $X_1(4m)_\C$ with
\begin{equation}\label{eqn:thetadiv}
\div(\Theta_m) = (\pi_2\circ d')^*\Sigma_{4,1} - (\pi_1\circ
d')^*\Sigma_{4,1}.
\end{equation}
Finally, we note that, as in the analytic case, $\Theta_m$ descends
through the map $$d':X_1(4m)\longrightarrow X_1(4,m)$$
(given by the
same transformation of moduli functors as above) exactly when $m$ is a
square.  To see this, note that the group $$(\Z/4m\Z)^\times \cong
(\Z/4\Z)^\times \times (\Z/m\Z)^\times$$
acts on the curve
$X_1(4m)_\C$ in the usual way by scaling the point of order $4m$.  One
sees by passage to the analytic world that the action of
$(\Z/m\Z)^\times$ on $\Theta_m$ by pullback is again given by the
character $(m/\cdot)$.

In the next section, we will need a few $q$-expansions of $\Theta_m$
in terms of Tate curves.  To compute these expansions, one must
translate this data into complex-analytic data and compute with
standard transformation formulas for $\theta$.  To do this, we proceed
as follows.  Let $\scr{R}$ denote the subring $\C((q_{4m})))$
consisting of $q$-expansions of holomorphic functions in the upper
half-plane $\H$ which are invariant under translation by $4m$, and for
$\tau\in\H$ let $\phi_\tau$ denote the map
\begin{eqnarray*}
  \scr{R} & \longrightarrow & \C \\
  \sum a_nq_{4m}^n & \longmapsto & \sum a_ne^{\frac{2\pi i n\tau}{4m}}.
\end{eqnarray*}

For any point $P$ of order $4m$ on $\Tate(q)$, the pair $(\Tate(q),P)$
is defined over $\Z[1/(4m),\zeta_4,\zeta_m]((q_{4m}))$, and if we pick
an embedding of the coefficient ring
$$\iota:\Z[1/(4m),\zeta_4,\zeta_m]\hookrightarrow \C$$
and base change
$(\Tate(q),P)$ to $\C((q_{4m}))$ accordingly, the result is actually
defined over $\scr{R}$ (simply look at the defining equations of
$\Tate(q)$). Moreover, if we then base change to $\C$ via the map
$\phi_\tau$, we arrive at a pair $(E_\tau,P')$ for some point $P'$ of
order $4m$ on $E_\tau$.  Thus
$$\phi_\tau(\iota(\Theta_m(\Tate(q),P))) = \Theta_m(E_\tau,P') =
\Theta_m(\gamma\tau)$$ where $$\gamma = \left(\begin{matrix} a
    & b \\ c & d\end{matrix}\right)\in\SL_2(\Z)$$ is chosen so that
the isomorphism 
\begin{eqnarray*}
  E_\tau & \stackrel{\sim}{\longrightarrow} & E_{\gamma\tau} \\
  z & \longmapsto & \frac{z}{c\tau+d}
\end{eqnarray*}
carries the point $P'$ to the unique point $P''$ on $E_{\gamma\tau}$
such that $4P'' = \gamma\tau/m$ and $mP'' = 1/4$.  Now
$\Theta_m(\gamma\tau)$ can be worked out explicitly from the
definition of $\Theta_m$, and the desired expansion is easily read
off.

We illustrate this technique by computing one of these $q$-expansions.
Fix a primitive $4^{\small\mathrm{th}}$ root of $1$, $\zeta_4$.  If
$\zeta_m$ is \emph{any} $m^{\small\mathrm{th}}$ root of $1$ (not
necessarily primitive), and $P$ is the unique point of order $4m$ with
$4P= \zeta_m q_m$ and $mP = \zeta_4$, then we claim that
$$\Theta_m(\Tate(q),P) = \frac{\sum_{n\in\Z} \zeta_m^{n^2}q_m^{n^2}} {
  \sum_{n\in \Z} q^{n^2}}.$$

We specify the embedding $\iota$ by sending $\zeta_4$ to $i$ and $\zeta_m$
to $e^{2\pi i k/m}$ for some integer $k$.  Thus
$$\phi_\tau(\iota(\Theta_m(\Tate(q),P))) = \Theta_m(E_\tau,P')$$ where
$P'$ is the unique point on $E_\tau$ such that $4P' = (k+\tau)/m$ and
$mP' = 1/4$.  One then notes that matrix $$\gamma =
\left(\begin{matrix} 1 & k \\ 0 & 1\end{matrix}\right)$$ has the
desired property.  That is, that the isomorphism 
\begin{eqnarray*}
  E_\tau & \stackrel{\sim}{\longrightarrow} & E_{\tau+k} \\
  z & \longmapsto & z
\end{eqnarray*}
sends $P'$ to the unique point on $E_{\tau+k}$ with $4P' = (k+\tau)/m
= \gamma\tau/m$
and $mP' = 1/4$.  Thus we are lead to $$\Theta_m(\tau+k) =
\frac{\theta(\frac{\tau+k}{m})} {\theta(\tau+k)} = \frac{\sum_{n\in\Z}
  e^{2\pi i n^2 \frac{\tau+k}{m}}} {\sum_{n\in\Z} e^{2\pi i
    n^2(\tau+k)}} = \frac{\sum_{n\in\Z} (e^{2\pi i k/m})^{n^2}(e^{2\pi
    i\tau/m})^{n^2} } { \sum_{n\in \Z} (e^{2\pi i\tau})^{n^2} }.$$
This implies that the $q$-expansion is as claimed, by the definitions
of $\iota$ and $\phi_\tau$.

Note in particular that this $q$-expansion with $\zeta_m=1$
corresponds to a $\Z[1/(4m)]$-rational cusp and has rational
coefficients.  By the $q$-expansion principle (in weight zero), we see
that $\Theta_m$ is in fact defined over $\Z[1/(4m)]$.  That is,
$\Theta_m$ extends to a rational function on $X_1(4m)_{\Z[1/(4m)]}$
with divisor given by (\ref{eqn:thetadiv}).

In the special case $m=l^2$ with $l$ and odd prime, we have two more
$q$-expansions.  The procedure used to derive these is exactly as
illustrated above.  We omit the details for brevity.  Recall that
$\Theta_{l^2}$ descends to $X_1(4,l^2)$ in this case.  Let
$\zeta_{l^2}$ be a primitive $(l^2)^{\small \rm th}$ root of unity.
Then we have
$$\Theta_{l^2}(\Tate(q),\zeta_4,\ip{\zeta_{l^2}}) =
l\frac{\sum_{n\in\Z} q^{l^2n^2}} {\sum_{n\in\Z} q^{n^2}}$$ and
$$\Theta_{l^2}(\Tate(q),\zeta_4,\ip{\zeta_{l^2}q_l}) =
\left(\frac{-1}{l}\right) \mathfrak{g}_l(\zeta_{l^2}^l)\frac{
  \sum_{n\in\Z} \zeta_{l^2}^{ln^2}q^{n^2}} {\sum_{n\in\Z} q^{n^2}}$$

Finally, we will need one more expansion in the special case $m=l$,
and odd prime.  Let $\zeta_l$ be a primitive $l^{\mathrm\small th}$
root of unity, and $P$ be the unique point of order $4l$ with
$4P=\zeta_l$ and $lP=\zeta_4$, then $$\Theta_l(\Tate(q),P) =
\mathfrak{g}_l(\zeta_l)\frac{ \sum_{n\in\Z} q^{ln^2}}{\sum_{n\in\Z}
  q^{n^2}}.$$

\section{Hecke Operators}\label{sec:hecke}

As explained in the Introduction, in order to construct Hecke
operators on the space  $M_{k/2}(4N,R)$, we look for rational
functions $\Theta$ on $X_1(4N,m)$ with the property that
$$\div(\Theta) = \pi_2^*\Sigma_{4N,k} - \pi_1^*\Sigma_{4N,k}.$$  
The following theorem is the manifestation in this formalism of the
classical fact that in half-integral weight, good Hecke operators
occur only for square indices.

\begin{theo}
  Let $m$ be a positive integer with $\gcd(4N,m)=1$, and let $R$ be a
  ring in which $4Nm$ is invertible. Then there exists a rational
  function $\Theta$ on $X_1(4N,m)_R$ with $$\div(\Theta) =
  \pi_2^*\Sigma_{4N,k} - \pi_1^*\Sigma_{4N,k}$$ if and only if $m$ is
  a square.  In case $m$ is a square, such a $\Theta$ is unique up to
  multiplication by a constant in $R^\times$.  
\end{theo}
\begin{proof}
  The uniqueness statement is clear.  Suppose that $\Theta$ has the
  specified divisor.  Then $d^*\Theta$ is a rational function on
  $X_1(4Nm)$ with divisor $$d^*(\pi_2^*\Sigma_{4N,k} -
  \pi_1^*\Sigma_{4N,k}) = (\pi_2\circ d)^*\Sigma_{4N,k} - (\pi_1\circ
  d)^*\Sigma_{4N,k}.$$
  By Proposition \ref{divtheta} and the
  compatibility of $\Sigma_{4N,k}$ with pullback through the standard
  map $$X_1(4Nm)\longrightarrow X_1(4m),$$
  this is the divisor of
  $\Theta_m^k$ on $X_1(4Nm)$.  Thus $\Theta$ and $\Theta_m^k$ agree up
  to a constant, and the ``only if'' statement follows from the fact
  that $\Theta_m^k$ descends through $d$ if and only if $m$ is a
  square (note that we are using that $k$ is odd here, as we must).
  The ``if'' statement follows simply because $\Theta_m^k$ furnishes a
  function on $X_1(4N,m)$ with the desired divisor when $m$ is a
  square.
\end{proof}

Let $l$ be a prime not dividing $4N$ and let $R$ be a ring in which
$4Nl$ is invertible.  We define $T_{l^2}$ to be the endomorphism of
$M_{k/2}(4N,R)$ defined by $$F \longmapsto
\frac{1}{l^2}\pi_{1*}(\pi_2^*F \cdot \Theta_{l^2}^k).$$
Note that this
endomorphism preserves the subspaces $M_{k/2}(4N,R,\chi)$ for each
character $\chi$.  The following result says that this definition
agrees with the classical one, as in (\cite{shimura}).
\begin{theo}
  Fix a primitive $(4N)^{\small\mathrm th}$ root of unity $\zeta_{4N}$
  and let $\infty$ denote the corresponding cusp.
  Let $F\in M_{k/2}(4N,R,\chi)$ have $q$-expansion $\sum a_nq^n$ at
  $\infty$.  Then the modular form $T_{l^2}F$ has $q$-expansion $\sum
  b_nq^n$ at $\infty$ where $$b_n = a_{l^2n} +
  \chi(l)\left(\frac{-1}{l}\right)^{\frac{k-1}{2}}
  l^{\frac{k-1}{2}-1}\left(\frac{n}{l}\right) a_n +
  \chi(l^2)l^{k-2}a_{n/l^2}.$$ 
\end{theo}
\begin{proof}
  We are given that $$F(\Tate(q),\zeta_{4N})
  \cdot \left(\sum_{n\in\Z}q^{n^2}\right)^k = \sum a_n q^n.$$
  The
  quantity $$\pi_{1*}(\pi_2^*F\cdot
  \Theta_{l^2}^k)(\Tate(q),\zeta_{4N})$$
  is a sum over the cyclic
  subgroups of order $l^2$ on $\Tate(q)$.  These subgroups are of
  three types, and we shall compute the corresponding contributions to
  the above sum one by one.  Choose a primitive root $\zeta_{l^2}$ as
  above.

  Firstly, there is the lone subgroup $\langle\zeta_{l^2}\rangle$.
  Passing to the quotient we find
  $$(\Tate(q)/\ip{\zeta_{l^2}},\zeta_{4N}/\ip{l^2}) \cong
  (\Tate(q^{l^2}), \zeta_{4N}^{l^2})$$ so that the contribution of
  this subgroup to the sum is
  \begin{eqnarray*}
    F(\Tate(q^{l^2}),\zeta_{4N}^{l^2}) 
    \Theta_{l^2}(\Tate(q),\zeta_{4N},\ip{\zeta_{l^2}})^k &=&
    F(\Tate(q^{l^2}), \zeta_{4N}^{l^2})\left(l\frac{\sum_{n\in\Z}
        q^{l^2n^2}}{\sum_{n\in\Z} q^{n^2}}\right)^k \\ && =
    \chi(l^2)l^k\frac{\sum 
      a_n q^{l^2n^2}}{(\sum_{n\in\Z}q^{n^2})^k}
  \end{eqnarray*}
  where we have used the $q$-expansion formulas for $\Theta_{l^2}$ form
  the previous section.
  
  Secondly, there are the subgroups $\ip{\zeta_{l^2}^iq_{l^2}}$ for
  $0\leq i\leq l^2-1$.  Note that
  $$(\Tate(q)/\ip{\zeta_{l^2}^iq_{l^2}},\zeta_{4N}/\ip{\zeta_{4N}^iq_{l^2}})
  \cong (\Tate(\zeta_{4N}^iq_{l^2}),\zeta_{4N}),$$ so that these
  subgroups collectively contribute 
  \begin{eqnarray*}
    \lefteqn{
    \sum_{i=0}^{l^2-1}F(\Tate(\zeta_{l^2}^iq_{l^2}),\zeta_{4N})
    \Theta_{l^2}(\Tate(q),\zeta_{4N},\ip{\zeta_{l^2}^iq_{l^2}})^k } && \\
    &=& \sum_{i=0}^{l^2-1}F(\Tate(\zeta_{l^2}^iq_{l^2}),\zeta_{4N})
    \left(\frac{\sum_{n\in\Z}\zeta_{l^2}^{in^2}q_{l^2}^{n^2}}
    {\sum_{n\in\Z} q^{n^2}}\right)^k \\ &=& \sum_{i=0}^{l^2-1}
    \frac{\sum a_n (\zeta_{l^2}^iq_{l^2})^n}{ (\sum_{n\in\Z}
    q^{n^2})^k} = l^2\frac{\sum a_{l^2n}q^n}{(\sum_{n\in\Z}
    q^{n^2})^k}.  
  \end{eqnarray*}

Lastly, there are the subgroups $\ip{\zeta_{l^2}^jq_{l}}$ for $1\leq
j\leq l-1$.  Note that
$$(\Tate(q)/\ip{\zeta_{l^2}^jq_l},\zeta_{4N}/\ip{\zeta_{l^2}^jq_{l}})
\cong (\Tate(\zeta_l^jq),\zeta_{4N}^l).$$  Thus the collective
contribution of these terms is 
\begin{eqnarray*}
  \lefteqn{
  \sum_{j=1}^{l-1}
  F(\Tate(\zeta_l^jq),\zeta_{4N}^l)
  \Theta_{l^2}(\Tate(q),\zeta_{4N},\ip{\zeta_{l^2}^jq_l})^k } && \\
  &=& \sum_{j=1}^{l-1} F(\Tate(\zeta_l^jq),\zeta_{4N}^l)\left(
  \left(\frac{-1}{l}\right)\mathfrak{g}_l(\zeta_l^j)
  \frac{\sum_{n\in\Z} \zeta_l^{jn^2}q^{n^2}}
  {\sum_{n\in\Z} q^{n^2}}\right)^k 
\end{eqnarray*}
Note that 
$$\mathfrak{g}_l(\zeta_l^j) =
\left(\frac{j}{l}\right) \mathfrak{g}_l(\zeta_l),$$
so that the above continues as 
\begin{eqnarray*}
\lefteqn{
  \chi(l)\left(\frac{-1}{l}\right)\left(\frac{\mathfrak{g}_l(\zeta_l)}
  {\sum_{n\in\Z}q^{n^2}}\right)^k 
  \sum_{j=1}^{l-1}
  \left(\frac{j}{l}\right)F(\Tate(\zeta_lq).\zeta_{4N}) \sum_{n\in\Z}(
  \zeta_l^jq)^{n^2} } && \\ &=&
  \chi(l)\left(\frac{-1}{l}\right)\left(\frac{\mathfrak{g}_l(\zeta_l)}
  {\sum_{n\in\Z}q^{n^2}}\right)^k
  \sum_{j=1}^{l-1} \left(\frac{j}{l}\right)\sum a_n (\zeta_l^jq)^n \\
  &=&
  \chi(l)\left(\frac{-1}{l}\right)\left(\frac{\mathfrak{g}_l(\zeta_l)}
  {\sum_{n\in\Z}q^{n^2}}\right)^k
  \sum_n a_n \left(\sum_{j=1}^{l-1} \left(\frac{j}{l}\right)
  \zeta_l^{jn}\right)q^n \\ &=&
  \chi(l)\left(\frac{-1}{l}\right)\frac{\mathfrak{g}_l(\zeta_l)^{k+1}}
  {\left(\sum_{n\in\Z}q^{n^2}\right)^k}\sum_n   
  \left(\frac{n}{l}\right) a_n q^n .
\end{eqnarray*}
It is well known that the Gauss sum above squares to
$$\mathfrak{g}_l(\zeta_l)^2 = \left(\frac{-1}{l}\right) l,$$
so we may
further continue the above as $$\chi(l)
\left(\frac{-1}{l}\right)^{\frac{k-1}{2}} \frac{
  l^{\frac{k+1}{2}}}{\left(\sum_{n\in\Z} q^{n^2}\right)^k} \sum_n
\left(\frac{n}{l}\right) a_nq^n.$$

Adding these three expressions together and dividing by $l^2$ we arrive
at the desired $q$-expansion for $T_{l^2}F$.
\end{proof}

Suppose now that $l$ is an odd prime dividing $N$.  Consider the map
$e$ defined as follows.
\begin{eqnarray*}
  e:  X_1(4N,l) & \longrightarrow & X_1(4l) \\
  (E,P,C) & \longmapsto & (E/C,(N/l)P/C).
\end{eqnarray*}
By chasing degeneracy maps around, one verifies that
$$e^*(\div(\Theta_l^{-1})) = \pi_2^*\Sigma_{4N,1} -
\pi_1^*\Sigma_{4N,1},$$ so that we get a well-defined endomorphism of
$M_{k/2}(4N,R)$, namely, $$U_l(F) = \frac{1}{l} \pi_{1*}(\pi_2^*F
\cdot e^*\Theta_l^{-k})$$

\begin{theo}
  Let $l$ be an odd prime dividing $N$.  Fix a primitive
  $(4N)^{\mathrm\small th}$ root of unity $\zeta_{4N}$ and let
  $\infty$ denote the corresponding cusp.  The endomorphism $U_l$ of
  $M_{k/2}(4N,R)$ has the effect $$\sum a_nq^n \longmapsto
  \mathfrak{g}_l(\zeta_{4N}^{4N/l})\sum a_{ln}q^n$$ on $q$-expansions
  at $\infty$.
\end{theo}
\begin{proof}
  Note that the collection of subgroups of $\Tate(q)$ of order $l$ not
  meeting the subgroup generated by $\zeta_{4N}$ is
  $\ip{\zeta_l^iq_l}$, for $0\leq i<l-1$.  We compute
  \begin{eqnarray*}
    \lefteqn{ \pi_{1*}(\pi_2^*F\cdot
      e^*\Theta_l^{-k})(\Tate(q),\zeta_{4N})} && \\  &=& 
    \sum_{i=0}^{l-1} (\pi_2^*F\cdot
    e^*\Theta_l^{-k})(\Tate(q),\zeta_{4N},\ip{\zeta_l^iq_l})  \\ &=& 
  \sum_{i=0}^{l-1}
  F(\Tate(q)/\ip{\zeta_l^iq_l},\zeta_{4N}/\ip{\zeta_l^iq_l})
  \Theta(\Tate(q)/\ip{\zeta_l^iq_l},\zeta_{4N}^{N/l}/\ip{\zeta_l^iq_l})^{-k}
      \\ &=& \sum_{i=0}^{l-1} F(\Tate(\zeta_l^iq_l),\zeta_{4N})
      \Theta(\Tate(\zeta_l^iq_l),\zeta_{4N}^{N/l})^{-k} \\ &=&
      \sum_{i=0}^{l-1}
      F(\Tate(\zeta_l^iq_l),\zeta_{4N})\left(\mathfrak{g}_l
      (\zeta_{4N}^{4N/l})\frac{\sum_{n\in\Z} q^{n^2}}{\sum_{n\in\Z}
      (\zeta_l^iq_l)^{n^2}}\right)^{-k} \\&=&
      \mathfrak{g}_l(\zeta_{4N}^{4N/l})^{-k} \sum_{i=0}^{l-1}
      \frac{\sum a_n (\zeta_l^iq_l)^n} {(\sum_{n\in\Z} q^{n^2})^k} =
      l\frac{\mathfrak{g}_l(\zeta_{4N}^{4N/l})}{ (\sum_{n\in\Z}
      q^{n^2})^k} \sum a_{ln}q^n
\end{eqnarray*}
dividing by $l$ we get the desired result.
\end{proof}

Note in particular that the effect of the map $U_l$ on $q$-expansions
depends on which $\infty$ is chosen, whereas that of its square
$(U_l)^2$ does not (since the square of the Gauss sum is independent
of the root of unity used).  On a related note, $U_l$ (unlike
$(U_l)^2$ or $T_{l^2}$) does not preserve the subspaces
$M_{k/2}(4N,R,\chi)$ but rather induces maps $$U_l: M_{k/2}(4N,R,\chi)
\longrightarrow M_{k/2}(4N,R,\chi\cdot(l/\cdot))$$ where $(l/\cdot)$
is the usual Legendre character.

\begin{rema}
  One can also construct operators $U_4$ and (if $N$ is even), $U_2$.
  Predictably, the operator $U_2$ multiplies the Nebentypus by
  $(2/\cdot)$.  One can also show that Kohnen's $+$-space makes good
  sense in this more general setting.  That is, one can construct a
  natural operator (closely related to $U_4$) on the space of forms, of
  which Kohnen's $+$-space is an eigenspace.  The details of these
  constructions are in the author's thesis (\cite {thesis}).
\end{rema}

\section{$p$-adic modular forms of half-integral weight}\label{sec:padic}

In this section we fix a rational prime $p\geq 5$. Let $K$ be a
complete subfield of $\C_p$ and let $r\in\Q\cap[0,1]$.  In this
section we will define the space of $M_{k/2}(4N,K,r)$ of
$r$-overconvergent $p$-adic modular forms of weight $k/2$ and level
$4N$, as well as a number of natural operators on this space.  For the
sake of simplicity of exposition, we restrict ourselves to the case
$p\nmid N$. 

By our choice of $p$, the weight $p-1$ and level $1$ Eisenstein series
$E_{p-1}$ furnishes a lifting the Hasse invariant to characteristic
zero.  It is well-known that $E_{p-1}$ has Fourier coefficients in
$\Q\cap\Z_p$ and therefore furnishes a section of $\omega^{p-1}$ on
$X_1(N)_{\Z_p}$ for any $N\geq 4$. Note that this makes sense even in
case $N=4$ because the bundle $\omega^2$ \emph{does} descend from the
moduli stack for the $\Gamma_1(4)$ problem to the coarse moduli scheme
$X_1(4)$ (even though $\omega$ \emph{does not}).

For $N\geq 4$, let $X_1(4N)^\an_K$ denote the rigid analytic space
over $\Q_p$ associated to the curve $X_1(4N)_K$.  For any $r$ as
above, we may also consider the region in this curve defined by the
inequality $|E_{p-1}|\geq r$.  This is known to be a
connected affinoid subdomain of $X_1(4N)_K^\an$, and will be denoted
by $X_1(4N)_{\geq r}^\an$.  Thus, in particular, $X_1(4N)_{\geq
  1}^\an$ is the ``ordinary locus'' in $X_1(4N)_K^\an$.  Similarly,
one may consider $X_1(4N,m)_{\geq r}^\an$.  This space is connected
exactly when $\gcd(m,p)=1$ (recall that we are assuming that $p\nmid
4N$).  For these facts and others which we will freely use in the
sequel, see \cite{coleman} and \cite{buzzard}.

\begin{defi}
  The space of $p$-adic modular forms of weight $k/2$, level $4N$
  ($p\nmid N$), and growth condition $r$ over $K$ is $$M_{k/2}(4N,K,r)
  = H^0(X_1(4N)_{\geq r}^\an,\Sigma_{4N,k}).$$
\end{defi}

Each of the spaces $M_{k/2}(4N,K,r)$ is a (generally
infinite-dimensional) Banach space over $K$.  As usual, forms in the
spaces above with $r<1$ are said to be \emph{overconvergent}.  The
space of all overconvergent forms is accordingly the inductive limit
$$M_{k/2}^\dagger(4N,K) = \lim_{r\to 1} M_{k/2}(4N,K,r).$$

The curve $\Tate(q)$ can also be thought of as a $p$-adic analytic
family of elliptic curves over the punctured unit disk $0<|q|<1$.  As
usual, in the presence of level $4N$ structure, we instead consider a
parameter $q_{4N}$ and define $q=q_{4N}^{4N}$. Then we may consider
the family $\Tate(q)$ over the punctured disk $0<|q_{4N}|<1$, so that
all of the $4N$-torsion is rational after adjoining the $4N^{\small\rm
  th}$ roots of unity.  

If $F\in M_{k/2}(4N,K,r)$ and $P$ is a point of order $4N$ on
$\Tate(q)$ then we may evaluate $F$ on the pair $(\Tate(q),P)$ to
obtain a meromorphic function on the punctured disk $0<|q_{4N}|<1$
which has a Laurent expansion in $\overline{K}((q_{4N}))$.  Just as in
the algebraic case, we can define the $q$-expansion of $F$ at this
pair $(\Tate(q),P)$ be adjusting this expansion by the corresponding
expansion of $\theta$ from Table \ref{tab:table3}.  Thus we find, just
as in the algebraic case, that a meromorphic function on
$X_1(4N)_{\geq r}^\an$ is in $M_{k/2}(4N,K,r)$ if and only if it is
analytic away from the cusps and all of these $q$-expansions are in
$\overline{K}[[q_{4N}]]$.

\section{$p$-adic Hecke operators}

Our aim is to define a Hecke action on the spaces $M_{k/2}(4N,K,r)$.
The situation is rather different depending on whether or not we are
dealing with the Hecke operator ``at $p$'', but in all cases we use a
variant of the ``pull-back, multiply by a unit, and push forward''
construction above relative to the usual correspondence $$\pi_1,\pi_2
: X_1(4N,m)_K^\an \rightrightarrows X_1(4N)_K^\an.$$

Suppose $Z_1$ and $Z_2$ are admissible opens in $X_1(4N)_K^\an$ and
$W$ is an admissible open in $X_1(4N,m)_K^\an$ such that the $\pi_i$
restrict to a pair of maps
$$\xymatrix{ & W\ar[dl]_{\pi_1}\ar[dr]^{\pi_2} \\ Z_1 & & Z_2}$$
Suppose further that $\pi_1^{-1}(Z_1)\subseteq \pi_2^{-1}(Z_2)$ and
$\Theta$ is a meromorphic function on $W$ with $$\div(\Theta)\geq
\pi_2^*\Sigma_{4N,k} - \pi_1^*\Sigma_{4N,k}.$$
These data furnish us
with a map $$H^0(Z_2,\Sigma_{4N,k})\longrightarrow
H^0(Z_1,\Sigma_{4N,k})$$
defined as the composition
$$\xymatrix{ H^0(Z_2,\Sigma_{4N,k})\ar[r]^-{\pi_2^*} &
  H^0(\pi_2^{-1}(Z_2),\pi_2^*\Sigma_{4N,k})\ar[r]^{\cdot\Theta} &
  H^0(\pi_2^{-1}(Z_2),\pi_1^*\Sigma_{4N,k}) \ar `r[d] `[l] `[dll]
  `[dl] [dl] 
\\ &
  H^0(\pi_1^{-1}(Z_1),\pi_1^*\Sigma_{4N,k})\ar[r]^-{\pi_{1*}} &
  H^0(Z_1,\Sigma_{4N,k})}$$ 
where the long arrow is restriction.

This setup suffices for the case $m=l^2$ with $\gcd(l,4Np)=1$.  Then we
may take $Z_1=Z_2=X_1(4N)_{\geq r}^\an$, $W=X_1(4N,l^2)_{\geq r}^\an$,
and $\Theta = \Theta_{l^2}^k$.  The two conditions that must be
satisfied, namely that the $\pi_i$ restrict as above and that
$\pi_1(Z_1)\subseteq \pi_2^{-1}(Z_2)$, together are equivalent to the
assertion that if $C$ is a cyclic subgroup of $E$ of order $l^2$, then
$|E_{p-1}(E)|\geq r$ if and only if $|E_{p-1}(E/C)|\geq r$.  Of
course, it suffices to verify the ``only if'' part since we may then
apply the result to the pair $(E/C,E[l^2]/C)$ to verify the other
implication.  At any rate, the result is well known and holds for
cyclic subgroups of all orders prime to $p$.

\begin{theo}
  There exists a continuous endomorphism $T_{l^2}$ of
  $M_{k/2}(4N,K,\chi,r)$ having the effect $$\sum a_nq^n\longmapsto
  \sum b_nq^n$$
  where $$b_n = a_{l^2n} +
  \chi(l)\left(\frac{-1}{l}\right)^{\frac{k-1}{2}}
  l^{\frac{k-1}{2}-1}\left(\frac{n}{l}\right) a_n +
  \chi(l^2)l^{k-2}a_{n/l^2}$$ on $q$-expansions at (any choice of)
  $\infty$.  
\end{theo}
\begin{proof}
  Dividing the endomorphism of $M_{k/2}(4N,K,r)$ obtained from the
  above choice of $\Theta$ by $l^2$ we arrive at a new endomorphism
  denoted $T_{l^2}$.  The effect of $T_{l^2}$ on $q$-expansions at
  $\infty$ is exactly as in the algebraic case, as the same
  computation verifies.
  
  We claim that $\|T_{l^2}\|\leq 1$, so that, in particular, $T_{l^2}$
  is continuous.  That the norms of $\pi_2^*$ and $\pi_{1*}$ are at
  most one are generalities.  It therefore suffices to show that
  multiplication by $\Theta_{l^2}^k$ as a map
  $$H^0(\pi_2^{-1}(Z_2),\pi_2^*\Sigma_{4N,k}) \longrightarrow
  H^0(\pi_2^{-1}(Z_2), \pi_1^*\Sigma_{4N,k})$$ is bounded by $1$.

  This follows immediately from the fact that $\Theta_{l^2}^k$ is an
  \emph{integral} section of the sheaf
  $$\OO_{X_1(4N,l^2)}(\pi_2^*\Sigma_{4N,k} - \pi_1^*\Sigma_{4N,k})$$
  on
  the algebraic curve $X_1(4N,l^2)_K$ (i.e., that it extends to a
  section over all of $X_1(4N,l^2)_{\OO_K}$) as the $q$-expansion
  principle readily demonstrates.
\end{proof}

The more interesting case to consider is that of $m=p^2$.  If
$r>p^{-1/p(1+p)}$, then the existence of the canonical subgroup of
order $p^2$ (see \cite{katz},\cite{buzzard}) furnishes a map
$$X_1(4N)_{\geq r}^\an \longrightarrow X_1(4N,p^2)_{\geq r}^\an$$
which is an isomorphism onto the connected component of right-hand
space containing the cusp associated to
$(\Tate(q),\zeta_{4N},\mu_{p^2})$ (for any primitive choice of
$\zeta_{4N}$), the inverse being $\pi_1$.  We will denote this
component by $W_{\geq r}$.  Standard results on quotienting by
canonical subgroups show that, for $r$ as above, $\pi_2$ restricts to
a map $$\pi_2: W_{\geq r}\longrightarrow X_1(4N)_{\geq r^{p^2}}^\an.$$

We will denote the restrictions of the various modular units and
divisors on $X_1(4N,p^2)^\an_K$ to $W_{\geq r}$ by their usual names.
Then we have, as usual, that $$\div(\Theta_{p^2}^{-k}) =
\pi_1^*\Sigma_{4N,k} - \pi_2^*\Sigma_{4N,k}.$$
Thus we have a
well-defined map
\begin{eqnarray*}
  M_{k/2}(4N,K,r) & \longrightarrow & M_{k/2}(4N,K,r^{p^2})\\
  F & \longmapsto & \ip{p^2}^*\pi_{2*}(\pi_1^*F\cdot (p\Theta_{p^2})^{-k})
\end{eqnarray*}

Let us determine the effect of this map on $q$-expansions.  To compute
the trace $\pi_{2*}$ we note that the fiber of $\pi_2$ above the pair
$(E,P)$ consists of all (isomorphism classes of) triple $(E',P',C)$
such that $C$ is the canonical subgroup of $E'$ of order $p^2$, and
$(E'/C,P'/C) \cong (E,P)$.  This is the same as the set of triples
$(E/C',mP/C',E[p^2]/C')$ where $C'$ runs through all order $p^2$ cyclic
subgroups of $E$ which have trivial intersection with \emph{its}
canonical subgroup of order $p^2$ (i.e. that do not contain the order
$p$ canonical subgroup of $E$), and $m$ is an inverse of $p^2$ mod
$4N$.  
So we may compute
\begin{eqnarray*}
\lefteqn{  (\ip{p^2}^*\pi_{2*}(\pi_1^*\ip{p^2}^*F\cdot
   (p\Theta_{p^2})^{-k}))(\Tate(q),\zeta_{4N}) 
   }&& \\ &=& 
  \sum_{i=0}^{p^2-1}
   (\pi_1^*F\cdot(p\Theta_{p^2})^{-k})(\Tate(q)/\langle 
  \zeta_{p^2}^iq_{p^2}\rangle,\zeta_{4N}/\langle\zeta_{p^2}^iq_{p^2}\rangle, 
  \Tate(q)[p^2]/\langle \zeta_{p^2}^iq_{p^2}\rangle) \\ &=&
   \sum_{i=0}^{p^2-1}
   F(\Tate(\zeta_{p^2}^iq_{p^2}),\zeta_{4N})
   (p\Theta_{p^2}(\Tate(\zeta_{p^2}^iq_{p^2}),\zeta_{4N},\mu_{p^2}))^{-k}
   \\ &=& \sum_{i=0}^{p^2-1}
   F(\Tate(\zeta_{p^2}^iq_{p^2}),\zeta_{4N})\left(
   \frac {\sum_{n\in\Z}  q^{n^2}} {\sum_{n\in\Z}
   (\zeta_{p^2}^iq_{p^2})^{n^2}} \right)^{-k} \\ &=& p^2\frac{\sum
   a_{p^2n}q^n} { (\sum_{n\in\Z} q^{n^2})^k}, 
\end{eqnarray*}
where $\sum a_nq^n$ is the $q$-expansion of $F$ at
$(\Tate(q),\zeta_{4N})$.  

\begin{theo}
  Suppose that  $p^{-1/p(1+p)}<r$.  Then there exists a continuous
  endomorphism $U_{p^2}$ of $M_{k/2}(4N,K,r)$ of norm at most $p^2$
  having the effect
  $$\sum a_nq^n\longmapsto a_{p^n}q^n$$
  on $q$-expansions at (any
  choice of) $\infty$.  Moreover, in case $r<1$, $U_{p^2}$ is in fact
  completely continuous.
\end{theo}
\begin{proof}
  The endomorphism we seek is the map 
  $$M_{k/2}(4N,K,r)\longrightarrow M_{k/2}(4N,K,r^{p^2})$$
  defined
  above divided by $p^2$ and post composed with the natural inclusion
  $$M_{k/2}(4N,K,r^{p^2}) \subseteq M_{k/2}(4N,K,r).$$
  As this latter
  map is completely continuous if $r<1$, it suffices to show that the
  map
  $$M_{k/2}(4N,K,r) \longrightarrow M_{k/2}(4N,K,r^{p^2})$$
  is
  continuous.  As in the case for $T_{l^2}$ above, it suffices to see
  that multiplication by the unit $(p\Theta_{p^2})^{-k}$ is bounded as
  a map $$H^0(W_{\geq r},\pi_1^*\Sigma_{4N,k}) \longrightarrow
  H^0(W_{\geq r},\pi_2^*\Sigma_{4N,k}).$$
  This is more subtle than the
  case of $T_{l^2}$ because of the bad reduction of $X_1(4N,p^2)$.
  
  Consider the modular unit $(p\Theta_{p^2})^{-1}$ on the algebraic
  curve $X_1(4N,p^2)_K$, and let $X_1(4N,p^2)_{\OO_K}$ denote the
  Katz-Mazur model of this curve over the ring of integers $\OO_K$ in
  $K$.  We claim that $(p\Theta_{p^2})^{-1}$ extends to all of
  $X_1(4N,p^2)_{\OO_K}$ as a section of the sheaf
  $$\OO_{X_1(4N,p^2)_{\OO_K}}(\pi_1^*\Sigma_{4N,1} -
  \pi_2^*\Sigma_{4N,1}).$$
  Since $X_1(4N,p^2)_{\OO_K}$ is Cohen-Macaulay
  over $\OO_K$ it suffices by the $q$-expansion principle to show that
  $(p\Theta_{p^2})^{-1}$ has integral $q$-expansions at a collection
  of cusps which, together, meet all three components of the special
  fiber of $X_1(4N,p^2)$ in characteristic $p$.  In particular, we fix
  primitive roots $\zeta_{4N}$ and $\zeta_{p^2}$ and take the three
  cusps corresponding to $(\Tate(q),\zeta_{4N},\langle \zeta_{p^2}\rangle)$,
  $(\Tate(q),\zeta_{4N},\langle \zeta_{p^2}q_p\rangle)$, and
  $(\Tate(q),\zeta_{4N}, \langle q_{p^2}\rangle)$.  That the indicated
  $q$-expansions are integral follows readily from the results of
  Section \ref{sec:thetam}.  

  It follows that multiplication by $(p\Theta_{p^2})^{-k}$ is of norm
  at most $1$, and therefore continuous.  Taking into account the
  division by $p^2$, we see that $\|U_{p^2}\|\leq p^2$.
\end{proof}

\bibliographystyle{smfplain}

\end{document}